\documentclass[12pt,a4paper]{amsart}
\usepackage{amsfonts}
\usepackage{amsthm}
\usepackage{amsmath}
\usepackage{amscd}
\usepackage[latin2]{inputenc}
\usepackage{t1enc}
\usepackage[mathscr]{eucal}
\usepackage{indentfirst}
\usepackage{graphicx}
\usepackage{graphics}
\usepackage{pict2e}
\usepackage{epic}
\numberwithin{equation}{section}
\usepackage[margin=2.9cm]{geometry}
\usepackage{epstopdf} 
\usepackage{tikz}

\usepackage{thmtools}
\usepackage{thm-restate}
\declaretheorem[name=Theorem,numberwithin=section]{thm}

\theoremstyle{plain}
\newtheorem{Th}{Theorem}[section]
\newtheorem{Lemma}[Th]{Lemma}

\newtheorem{Prop}[Th]{Proposition}

 \theoremstyle{definition}
\newtheorem{Def}[Th]{Definition}

\newtheorem{?}[Th]{Problem}

\newtheorem*{thm*}{Theorem}

\begin{document}

\title{Some Ordered Ramsey Numbers of Graphs on Four Vertices}

\author[W.~Overman]{Will Overman}
\address{California Institute of Technology \\ Department of Mathematics \\
Pasadena CA} 
\email{woverman@caltech.edu}

\author[J.~Alm]{Jeremy F.~Alm}
\address{Lamar University \\ Department of Mathematics \\
Beaumont TX} 
\email{jalm@lamar.edu}

\author[K.~Coffey]{Kayla Coffey}
\address{Stephen F.~Austin University \\ Department of Mathematics \\
Nacodoches TX} 
\email{kcoffey621@gmail.com}

\author[C.~Langhoff]{Carolyn Langhoff}
\address{Lamar University \\ Department of Computer Science  \\
Beaumont TX} 
\email{clanghoff@lamar.edu }


\begin{abstract} An ordered graph $H$ on $n$ vertices is a graph whose vertices have been labeled bijectively with $\{1,...,n\}$. The ordered Ramsey number $r_<(H)$ is the minimum $n$ such that every two-coloring of the edges of the complete graph $K_n$ contains a monochromatic copy of $H$ such that the vertices in the copy appear in the same order as in $H$. 

Although some bounds on the ordered Ramsey numbers of certain infinite families of graphs are known, very little is known about the ordered Ramsey numbers of specific small graphs compared to how much we know about the usual Ramsey numbers for these graphs. In this paper we tackle the problem of proving non-trivial upper bounds on orderings of graphs on four vertices. We also extend one of our results to $n+1$ vertex graphs that consist of a complete graph on $n$ vertices with a pendant edge to vertex 1. Finally, we use a SAT solver to compute some numbers exactly. 
\end{abstract}

\maketitle

\section{Introduction} For a given graph $H$, the Ramsey number $r(H)$ is defined to be the smallest integer $n$ such that for any two-coloring of the edges of the complete graph on $n$ vertices, $K_n$, we can find a monochromatic copy of $H$. In this paper we will consider $\textit{ordered Ramsey numbers}$, which are an analogue of Ramsey numbers for ordered graphs. The systematic study of ordered Ramsey numbers began with a 2014 paper by Conlon, Fox, Lee, and Sudakov \cite{conlon}.

An ordered graph $H$ on $n$ vertices is a graph whose vertices have been labeled bijectively with $\{1,...,n\}$. We say that an ordered graph $G$ on $N$ vertices \textit{contains} an ordered graph $H$ on $n$ vertices if there is a map $\phi: [n] \to [N]$ such that $\phi(i) < \phi(j)$ for $1 \leq i < j \leq n$ and such that if $(i,j) \in E(H)$, then $(\phi(i), \phi(j)) \in E(G)$. \cite{conlon} Thus the containment is order-preserving in the sense that given a copy of $H$ in $G$ the lowest ordered vertex (by the ordering of $G$) in the copy must correspond to the vertex labeled 1 in $H$ and so on. For example, if $H$ is the cycle on four vertices with labeling $\{1,2,3,4\}$ where $E(H)=\{(1,2),(2,3),(3,4),(4,1)\}$, then a possible monochromatic copy of $H$ in some larger graph $G$ could be on vertices $\{2,5,7,9\}$ with monochromatic edges $\{(2,5),(5,7),(7,9),(9,2)\}$. 

Then we can define the $\textit{ordered Ramsey number}$ , $r_< (H)$, of an ordered graph $H$ to be smallest integer $n$ such that for any ordering and any two-coloring of $K_n$ we can find a monochromatic, order-preserving copy of $H$ contained in $K_n$. Recall that a \textit{coloring} by $m$ colors of the edges $E(G)$ of a graph $G$ is a  map $c: E(G) \to [m]$. The first observation we can make about ordered Ramsey numbers is that for any ordering of a graph $H$, we clearly have $r(H) \leq r_<(H)$ where $r(H)$ is the usual Ramsey number of an unordered $H$ and $r_<(H)$ is the ordered Ramsey number. This gives us a trivial lower bound. Also observe that the trivial upper bound for the ordered Ramsey number of an ordered graph $H$ on $n$ vertices is the usual Ramsey number of $K_n$. This follows from the following lemma, whose truth is easy to see. 
\begin{Lemma}
An ordered monochromatic complete graph on $n$ vertices necessarily contains an ordered copy of any ordered graph on $n$ vertices, regardless of its ordering.
\end{Lemma}
This is clear since for any ordered graph on $n$ vertices we can find vertices in $K_n$ with the ordering we want and we already know all the edges are monochromatic. Despite its simplicity, Lemma 1.1 will be useful in many of our proofs.

The paper by Conlon et al. proved a number of results for ordered Ramsey numbers of certain infinite families of graphs \cite{conlon}.  Balko et al. established results for ordered Ramsey numbers on particular orderings of certain graph families such as paths, stars, and cycles \cite{balko}. Thus we will not investigate these graphs on four vertices in this paper. So far the only paper focusing on proving ordered Ramsey results for small graphs was by Chang in \cite{chang}. In that paper, Chang proved upper bounds for Ramsey numbers of 1-orderings for graphs on 4 vertices. A 1-ordering of a graph $H$ on $n$ vertices consists of a labeling of just one vertex with some integer from $\{1,...,n\}$. Then a copy of $H$ in some ordered complete graph just needs to preserve the ordering of this given vertex. Here we will focus on complete orderings of graphs on 4 vertices. 

In this paper we will prove upper bounds on the Ramsey numbers for certain total orderings of graphs on four vertices. Specifically, in Section \ref{sec:K2} we will examine orderings of $K_2 \cup K_2$, in Section \ref{sec:K4e} we examine orderings of the diamond graph, and in Section \ref{sec:pan} we examine the 3-pan graph. In section \ref{sec:pendant} we extend our upper bound of the 3-pan graph to the infinite family of complete graphs with a pendant edge. Specifically, we will prove the following, where the definition of a ``complete with 1-pendant'' graph is given in Section \ref{sec:pendant}.

\begin{thm}
The ordered Ramsey number of the complete with 1-pendant graph on $n+1$ vertices is $R(n)+2n-1$.
\end{thm}

Note that upper bounds on some orderings immediately give the same upper bound on ``symmetric'' orderings. By ``symmetric'' orderings, we mean that if we had a graph with vertices $a,b,c,d$ labeled $1,2,3,4$, then an upper bound on this ordering of the graph would also apply to the ordering $4,3,2,1$ by just ``flipping'' the argument. We will not explicitly mention when this symmetry applies to our results, but it is possible to apply it to a number of our results.

\section{Ordered Ramsey Numbers of $K_2 \cup K_2$}\label{sec:K2}

The proofs for upper bounds on orderings of $K_2 \cup K_2$ will be relatively straightforward, but hopefully illustrative of techniques we will use on other graphs. Also, we will be able to exhibit constructions showing that our lower bounds are tight for some orderings of $K_2 \cup K_2$, thus completely determining the ordered Ramsey number for those orderings. 

\begin{figure}[h]
\centering
\includegraphics[width=6cm]{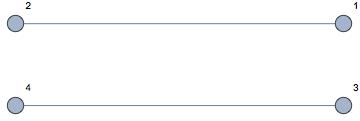}
\caption{Ordering $A$ of $K_2 \cup K_2$}
\end{figure}
We will first investigate the ordering of $K_2 \cup K_2$ given in Figure 1, which we will refer to as ordering $A$.

\pagebreak

\begin{Prop}
The ordered Ramsey number of $K_2 \cup K_2$ with ordering $A$ is 6.
\end{Prop}

\begin{proof} We will first show that $r_{<_A}(K_2 \cup K_2) \leq 6$. Recall that a \textit{coloring} by $m$ colors of the edges $E(G)$ of a graph $G$ is a surjective map $c:[m] \to E(G)$. So when we say ``color'' an edge, we mean choosing some color $\{1,...,m\}$ to assign to the edge, and here specifically we only deal with the case where $m=2$ and we have colors red and blue. Without loss of generality, color the edge between vertices 1 and 2 red. Then we know that all of the edges between vertices 3 through 6 must be blue, or else we would have an ordered copy of $K_2 \cup K_2$. But then we have a complete blue graph on four vertices, and thus by Lemma 1.1 we have a monochromatic copy of $K_2 \cup K_2$. So $r_{<_A}(K_2 \cup K_2) \leq 6$.

Now we will demonstrate that we can find a complete graph on 5 vertices with an edge coloring that does not produce a copy of $K_2 \cup K_2$. Color every edge from 1 to $\{2,3,4,5\}$ red. Then color the edge from 2 to 3 red. Color every other edge blue. See Figure 2. We claim that this graph does not have a monochromatic copy of $K_2 \cup K_2$. There is no red $K_2 \cup K_2$ since every red edge except $(2,3)$ originates from 1, and thus to find a copy of $K_2 \cup K_2$ with ordering $A$ we would need to find an edge with a higher lowest vertex than 1, which is only given by $(2,3)$, but there is no copy of $K_2 \cup K_2$ with $(2,3)$ either. So there is no red $K_2 \cup K_2$. 

We can also see that there is no blue $K_2 \cup K_2$. Since every edge from 1 is red, a monochromatic copy of $K_2 \cup K_2$ would have to be within the subgraph induced by the vertices $\{2,3,4,5\}$. The only possible way to get a blue copy of $K_2 \cup K_2$ on these vertices is if we have a blue $(2,3)$ and a blue $(4,5)$, which is not the case. So $r_{<_A}(K_2 \cup K_2) > 5$, which proves our result $r_{<_A}(K_2 \cup K_2) = 6$.
\end{proof}
\begin{figure}[h]
\centering
\includegraphics[width=6cm]{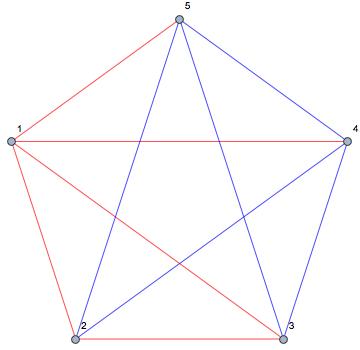}
\caption{Colored $K_5$ with no $K_2 \cup K_2$ having ordering $A$}
\end{figure}

Now we will examine the ordering on $K_2 \cup K_2$ given in Figure 3, which we will refer to as ordering $B$. 
\begin{figure}[h]
\centering
\includegraphics[width=5cm]{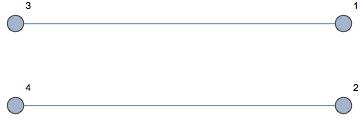}
\caption{Ordering $B$ of $K_2 \cup K_2$}
\end{figure}

\begin{Prop}
The ordered Ramsey number of $K_2 \cup K_2$ with ordering $B$ is 5.
\end{Prop}

\begin{proof}
Note that we trivially have that $r_{<_B}(K_2 \cup K_2) \geq 5$ since the usual Ramsey number of $K_2 \cup K_2$ is 5 \cite{small}, and we noted in the introduction that $r(H) \leq r_<(H)$ for any ordering of $H$. So we only need to prove that $r_{<_B}(K_2 \cup K_2) \leq 5$.

Without loss of generality, color the edge $(1,3)$ red. This forces the edges $(2,4)$ and $(2,5)$ to both be blue. And since $(2,5)$ is blue, the edge $(1,4)$ is forced to be red, which then forces the edge $(3,5)$ to be blue. However, this gives us a blue copy of $K_2 \cup K_2$ with ordering $B$ having edges $(2,4)$ and $(3,5)$. So $r_{<_B}(K_2 \cup K_2) \leq 5$, which proves $r_{<_B}(K_2 \cup K_2) = 5$.
\end{proof}

The last ordering of $K_2 \cup K_2$ to consider is given in Figure 4, which we will refer to as ordering $C$. 
\begin{figure}[h]
\centering
\includegraphics[width=6cm]{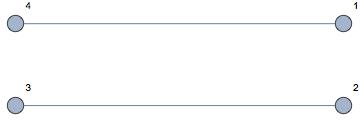}
\caption{Ordering $C$ of $K_2 \cup K_2$}
\end{figure}

\begin{Prop}
The ordered Ramsey number of $K_2 \cup K_2$ with ordering $C$ is 6.
\end{Prop}

\begin{proof}
First we will show that $r_{<_C}(K_2 \cup K_2) \leq 6$. Without loss of generality color edge $(1,4)$ red. Then this forces $(2,3)$ to be blue. Then $(2,3)$ being blue, forces edges $(1,5)$ and $(1,6)$ to be red. Then edge $(1,6)$ being red forces edges $(2,4)$ and $(2,5)$ to be blue. But now any coloring of $(3,4)$ will gives us a monochromatic copy of $K_2 \cup K_2$. If $(3,4)$ is red, then it forms a copy with $(1,6)$, while if $(3,4)$ is blue, then it forms a copy with $(2,5)$. Thus we have that $r_{<_C}(K_2 \cup K_2) \leq 6$.

Now we will show that $r_{<_C}(K_2 \cup K_2) > 5$ by exhibiting an ordered $K_5$ with no copy of $K_2 \cup K_2$ having ordering $C$. Color every edge from 1 red and color every edge from 5 red. Then color the triangle formed by $\{2,3,4\}$ blue. This clearly doesn't have a blue copy. To see that it doesn't have a red copy, note that since every red edge involves either 1 or 5, the only way to get a copy of $K_2 \cup K_2$ with ordering $C$ is if the copy involves the edge $(1,5)$, since the lowest and highest vertices in a copy must share an edge. But then all of the edges between $\{2,3,4\}$ are blue, so there is no red copy with edge $(1,5)$.  Thus $r_{<_C}(K_2 \cup K_2) > 5$, so we get our result $r_{<_C}(K_2 \cup K_2) = 6$.
\end{proof}

\begin{figure}[h]
\centering
\includegraphics[width=6cm]{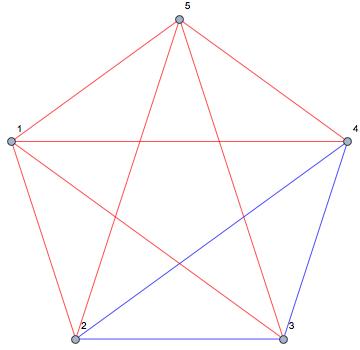}
\caption{Colored $K_5$ with no $K_2 \cup K_2$ having ordering $C$}
\end{figure}

\section{Ordered Ramsey Numbers of the Diamond Graph} \label{sec:K4e}

The diamond graph can be considered as $K_4-e$, i.e. the complete graph on four vertices with an edge removed. The usual Ramsey number for the diamond graph is 10 \cite{small}, so this is the trivial lower bound on the ordered Ramsey number of the diamond graph for any ordering. Also recall that the trivial upper bound for a graph on four vertices is $R(4)=18$. In \cite{chang}, Chang obtained upper bounds between 13 and 17 for 1-orderings of $K_4-e$. He also demonstrated that the lower bound for ordering $A$ (see Figure 6) of $K_4-e$ is at least 12, but was thought to be higher since his program was able to find 25536 constructions of $K_4-e$ with ordering $A$ on 11 vertices. In Section \ref{sec:SAT}, we will see that the correct number is 15.

\begin{figure}[h]
\centering
\includegraphics[width=6cm]{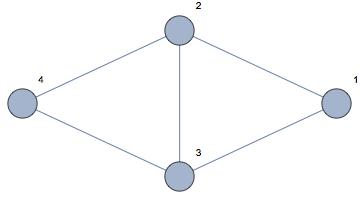}
\caption{Ordering $A$ of $K_4-e$}
\end{figure}

We will begin by proving that $r_{<_A}(K_4-e) \leq 17$. Our proof will rely on the fact that $R(K_3,K_4)=9$. \cite{small} Where we recall that $R(G,H)$ refers to the minimum integer $n$ such that for any edge-coloring of the complete graph on $n$ vertices we will get either a red $H$ or a blue $G$. So in this case, this means that for any edge coloring of $K_9$, we will either get a red $K_3$ or a blue $K_4$. Also recall the trivial fact that $R(K_4,K_3)=R(K_3,K_4)$. 

\begin{Th}
The ordered Ramsey number of $K_4-e$ with ordering $A$ is bounded above by 17.
\end{Th}

\begin{proof}
Consider all of the edges from vertex 17 in the complete graph $K_{17}$. There are 16 such edges. Assume that $x\geq 9$ of them are the same color, which, without loss of generality, we can assume to be red. Then consider the set $X$ of the $x$ vertices connected to 17 by a red edge. Since $x \geq 9$, the subgraph of $K_{17}$ induced by these vertices contains either a red $K_3$ or a blue $K_4$. If there is a blue $K_4$, then we have a copy of $K_4-e$ with ordering $A$ by Lemma 1.1, so we can assume that we have a red $K_3$ instead. But then all three vertices in this red triangle also share a red edge to vertex 17, which implies that we have a red $K_4$. Thus we again would get a monochromatic copy of $K_4-e$ with ordering $A$. So there cannot be 9 or more edges of the same color form vertex 17.

Thus we can assume that there are exactly 8 red and 8 blue edges from vertex 17. Let the set of 8 vertices connected to 17 by a red edge be $X$ and let the set of vertices connected to 17 by a blue edge be $Y$. Vertex 1 is either in $X$ or $Y$; assume, without loss of generality, that $1 \in Y$. Then take the set $Z=\{1\} \cup X$. Then $|Z|=9$, so again we either have a red triangle or a blue $K_4$. We can assume again that there is no blue $K_4$, so there must be a red $K_3$ in $Z$. If the vertices of the red $K_3$ are in $X$, then we get a red copy of $K_4$ by considering the edges from vertex 17, so we're done. Thus vertex 1 must be in the red triangle. Let $p,q \in Y$ be the other two vertices in the red triangle. Then we know that vertex 1 has red edges to $p$ and $q$, and that $p$ and $q$ have a red edge between each other, and finally that $p$ and $q$ have red edges to vertex 17 since $p,q \in X$. Thus we get a red copy of $K_4-e$ with ordering $A$. Thus we have that $r_{<_A}(K_4-e) \leq 17$
\end{proof}

\begin{figure}[h]
\centering
\includegraphics[width=6cm]{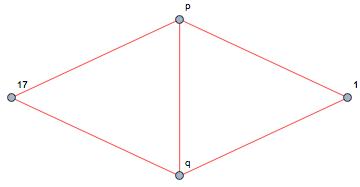}
\caption{Red copy of $K_4-e$ with $p,q \in X$}
\end{figure}

Now we will consider the ordering of the diamond graph given by Figure 8, which we will refer to as ordering $B$. 

\begin{figure}[h]
\centering
\includegraphics[width=6cm]{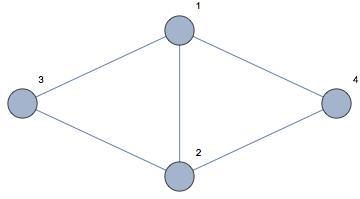}
\caption{Ordering $B$ of $K_4-e$}
\end{figure}

\begin{Th}
The ordered Ramsey number of $K_4-e$ with ordering $B$ is bounded above by 15.
\end{Th}

\begin{proof}
Consider the complete graph $K_n$ where $n$ is yet to be determined. Without loss of generality, assume that the edge $(1,2)$ is colored red. Now we will define four sets that partition the remaining vertices $\{3,4,...,n\}$. Let $RR$ be the set of vertices that have a red edge from 1 and a red edge from 2. Let $BB$ be the set of vertices that have a blue edge from 1 and a blue edge from 2. Let $RB$ be the set of vertices that have a red edge from 1 and a blue edge from 2. Finally, let $BR$ be the set of vertices that have a blue edge from 1 and a red edge from 2. Note that these four sets form a partition of all vertices $\{3,4,...,n\}$. Now assume that we do not have a monochromatic copy of $K_4-e$ with ordering $B$.

Clearly we have that $|RR| \leq 1$ since otherwise we would get a copy of $K_4-e$ with ordering $B$ since all the vertices in this set are necessarily greater than 1 or 2, so if 1 and 2 both had red edges to more than other vertex we would get a copy.

Now consider $RB$. We claim that $|RB| \leq 3$. To see this, assume that $|RB|=4$ and note that the vertices in $RB$ must have some total ordering. Without loss of generality, order them $3,4,5,6$. Now we know that out of the edges $(3,4),(3,5),(3,6)$ at least two of them must be the same color. Let these two edges of the same color be $(3,x)$ and $(3,y)$ with $x<y$. Then regardless of which color these two edges are, we get a monochromatic copy of $K_4-e$ with ordering $B$ since if they are red, then the vertices $\{1,3,x,y\}$ form a red copy, while if they are blue, then the vertices $\{2,3,x,y\}$ form a blue copy (see Figures 9 and 10). Thus we have that $|RB| \leq 3$. And a completely analogous argument shows that $|BR| \leq 3$.

So finally we need to consider $BB$. We claim that $|BB| \leq 5$. To see this, assume $|BB|=6$. The vertices in $BB$ are totally ordered, so without loss of generality, number them $3,4,5,6,7,8$. Note that of the edges $(3,4),(3,5),(3,6),(3,7),(3,8)$ only one of them can be blue since if we had two blue edges $(3,x)$ and $(3,y)$, then we'd get a copy of $K_4-e$ with ordering $B$ on vertices $\{1,3,x,y\}$.  

First we assume that one vertex from $\{4,5,6,7,8\}$ does have a blue edge from 3. Let it be vertex $x$. Then consider the subgraph on the vertices $Q=\{4,5,6,7,8\} \setminus \{x\}$. Let $y$ be the lowest ordered vertex in $Q$. Then we know that $y$ has an edge to each of the other three vertices in $Q$ and thus at least two of these edges are the same color. If there are two blue edges, call them $(y,y')$ and $(y,y'')$, then we get a copy of $K_4-e$ with ordering $B$ on vertices $\{1,y,y',y''\}$ since $y < y'$ and $y < y''$. If there are two red edges, call them $(y,y')$ and $(y,y'')$, then we get a copy on the vertices $\{3,y,y',y''\}$ where we recall that 3 has red edges to every vertex except $x$, which is not in $Q$. So if there is a blue edge then we see that we get a copy of $K_4-e$ with ordering $B$ if $|BB| \leq 5$.

Now consider the case in which 3 has red edges to all of $\{4,5,6,7,8\}$. Then we can clearly see that we can use the same argument as the case in which we do have a blue edge since we still have at least four vertices to which 3 has a red edge. In fact we can just "forget" vertex 8. Then we can see that vertex 4 either has two red or two blue edges to $\{5,6,7\}$, so we will get a monochromatic copy of $K_4-e$ with ordering $B$ either using vertex 1 if it's a blue copy or vertex 3 if it's a red copy. So we have that $|BB| \leq 5$.

Thus we have that in order to avoid a monochromatic copy of $K_4-e$ with ordering $B$ we need $n$ to be less than or equal to $2+|RR|+|RB|+|BR|+|BB| \leq 2+1+3+3+5=14$. Any vertex we add to the graph will have to go in one of $RR$, $BB$, $RB$, or $BR$, which would thus give us a monochromatic $K_4-e$ with ordering $B$. So $r_{<_B}(K_4-e) \leq 15$.
\end{proof}
We will see in Section \ref{sec:SAT} that 12 is the correct Ramsey number. 

\begin{figure}[h]
\centering
\includegraphics[width=6cm]{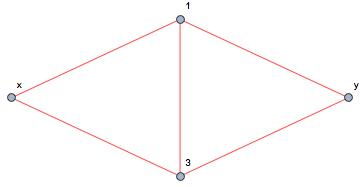}
\caption{Red copy of $K_4-e$ with ordering $B$}
\end{figure}

\begin{figure}[h]
\centering
\includegraphics[width=6cm]{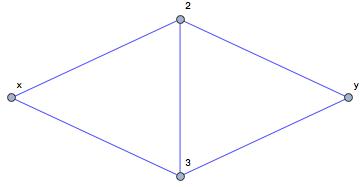}
\caption{Blue copy of $K_4-e$ with ordering $B$}
\end{figure}

The last ordering of the diamond graph we will consider is the one given in Figure 11, which we will refer to as ordering $C$.

\begin{figure}[h]
\centering
\includegraphics[width=6cm]{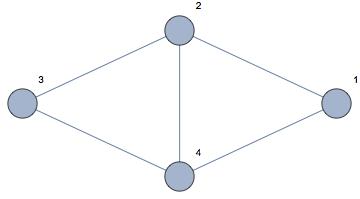}
\caption{Ordering $C$ of $K_4-e$}
\end{figure}

In order to establish an upper bound on this ordering, we will first need two lemmas concerning the ordered graph on three vertices with edges $E=\{(1,2),(2,3)\}$\, i.e. the ordered path on 3 vertices. Denote this graph $P_3^<$. 

The following two lemmas follow from Lemma 18 in \cite{balko}, but are given short proofs here for completeness. 
\begin{Lemma}
Any edge coloring of the ordered complete graph on 5 vertices either contains a red copy of $P_3^<$ or a blue copy of $K_3$, i.e. $r_<(P_3^<, K_3) \leq 5$.
\end{Lemma}

\begin{proof}
Consider $K_5$ and consider the subgraph on vertices $\{1,2,3\}$. Then we know there must be some red edge, call it $e_1=(x_1,x_2)$. Now consider the subgraph on $\{x_2,4,5\}$, then there must also be some red edge on this subgraph. But we know that this edge cannot involve vertex $x_2$, or else we would get a red copy of $P_3^<$. Thus we must have that $(4,5)$ is red. Then this implies that $(1,4),(2,4),(3,4)$ must all be blue. But then this implies that $(1,2)$ must be red or else we'd get a blue $K_3$ on $\{1,2,4\}$ and also implies that $(2,3)$ must be red or we'd get a blue $K_3$ on $\{2,3,4\}$. But then we have that $(1,2)$ and $(2,3)$ are both red, so we get a red copy of $P_3^<$. Thus $r_<(P_3^<, K_3) \leq 5$.
\end{proof}

\begin{Lemma}
Any edge-coloring of the ordered complete graph on 7 vertices either contains a red $P_3^<$ or a blue $K_4$, i.e. $r(P_3^<,K_4) \leq 7$. 
\end{Lemma}

\begin{proof}
Consider the complete ordered graph of $K_7.$ Consider vertex 4. Vertex 4 must have a blue edge connected to each vertex in the set $\{1,2,3\}$ or $\{5,6,7\}$ or else we would get a red copy of $P_3^{<}.$ Without loss of generality, assume vertex 4 has blue edges to the set $\{1,2,3\}.$ Now, among the edges $(1,2),(1,3)$ and $(2,3)$ at least one edge has to be red in order to avoid a blue copy of $K_4.$ There must also be a blue edge among the edges $(1,2),(1,3)$ and $(2,3)$ in order to avoid a red copy of $P_3^{<}.$ Assume the set $\{x_1, x_2, x_3\}$ represents the set of vertices $\{1,2,3\}$ in some order. Let the edge $(x_1,x_2)$ with $x_1 < x_2$ be red. Let the edge $(x_2,x_3)$ be blue. Now the vertex $x_2$ must have blue edges to each vertex in the set $\{5,6,7\}$ in order to avoid a red copy of $P_3^{<}.$ Now, among the edges $(5,6),(5,7)$ and $(6,7),$ one must be red in order to avoid a blue copy of $K_4.$ Assume the set $\{y_1,y_2,y_3\}$ represents the set $\{5,6,7\}$ in some order. Let the edge $(y_1,y_2)$ be red with $y_1 < y_2.$ Now, each vertex in the set $\{5,6,7\}$ must have at least one red edge to the blue triangle $(x_2,x_3,4)$ in order to avoid a blue copy of $K_4.$ Thus, creating a red copy of $P_3^{<}$ from a vertex in the triangle $(x_2,x_3,4)$ and the edge $(y_1,y_2).$ Therefore, $r(P_3^{<}, K_4) \leq 7.$ 
\end{proof}


\begin{Th}
The ordered Ramsey number of $K_4-e$ with ordering $C$ is bounded above by 14.
\end{Th}

\begin{proof}
Consider $K_{14}$. Consider the edges from vertex 14. Let $B$ be the set of vertices to which 14 has blue edges, and let $R$ be the set of vertices to which 14 has red edges. Clearly either $R$ or $B$ has size greater than or equal to 7. Assume, without loss of generality, that $|R| \geq 7$. Then since $|R|\geq 7$ we know by Lemma 3.4 that $R$ either has a red $P_3^<$ or a blue $K_4$. If we have a blue $K_4$, then we're done. But then if we have a red $P_3^<$, we get a red copy of $K_4-e$ with ordering $C$ since 14 is connected to all three vertices in the copy of $P_3^<$ by a red edge. 
\end{proof}

We will see in Section \ref{sec:SAT} that 14 is indeed the correct Ramsey number.

\section{Ordered Ramsey Numbers of the 3-Pan Graph}\label{sec:pan}

All of the orderings of the 3-pan graph that we will consider will have a pendant edge from the triangle on vertices $\{2,3,4\}$ to vertex 1. We will consider the orderings we get from attaching vertex 1 to all three possible vertices of the triangle. Note that the usual Ramsey number of the 3-pan is 7, so that is our trivial lower bound. First we will investigate the ordering in which vertex 1 is attached to vertex 4, which we will refer to as ordering $A$.

\begin{figure}[h]
\centering
\includegraphics[width=6cm]{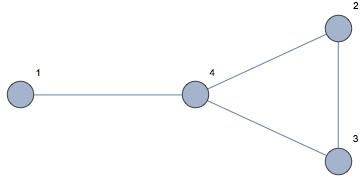}
\caption{Ordering $A$ of the 3-pan}
\end{figure}

\begin{Th}
The ordered Ramsey number of the 3-pan with ordering $A$ is bounded above by 10.
\end{Th}

\begin{proof}
Consider $K_{10}$. Vertex 10 must have at least 5 edges that are red or 5 edges that are blue to the other 9 vertices. Assume, without loss of generality, that there are 5 red edges. Let $R$ be the set of $\geq 5$ vertices with red edges to vertex 10. Since all of the vertices of $K_{10}$ are totally ordered, the vertices in $R$ are totally ordered. Remove the $|R|-4$ lowest vertices from $R$ so that we are left with the four vertices in $R$ with the highest ordering. Now if there is any red edge $e=(x_1,x_2)$ with $x_1 < x_2$ amongst these four vertices, then we get a red triangle on $\{x_1,x_2,10\}$, and since 10 has a red edge to at least one other vertex with ordering less than both $x_1$ and $x_2$, we get a red copy of the 3-pan with ordering $A$. But if we don't have a red edge amongst the four highest vertices in $R$, then we get a blue $K_4$, so by Lemma 1.1 we get a blue copy of the 3-pan with ordering $A$. Thus $r_{<_A}(\text{3-pan}) \leq 10$.
\end{proof}

Next we will consider the ordering we get by attaching vertex 1 to vertex 3. We will refer to this as ordering $B$ of the 3-pan.
\begin{figure}[h]
\centering
\includegraphics[width=6cm]{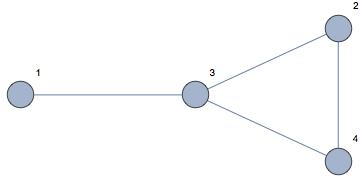}
\caption{Ordering $B$ of the 3-pan}
\end{figure}

\begin{Th}
The ordered Ramsey number of the 3-pan with ordering $B$ is bounded above by 14.
\end{Th}

\begin{proof}
Consider $K_{14}$. Now just consider the subgraph on vertices $\{9,10,11,12,13,14\}$. We know that any complete graph on 6 vertices contains a monochromatic triangle \cite{small}, so assume there is a red triangle $\{x,y,z\}$ with $x < y < z$ in this subgraph. Then in order to avoid a red copy of the 3-pan with ordering $B$, we must have that vertex $y$ has a blue edge to all of the vertices $\{1,2,3,4,5,6,7,8\}$.

Now consider the subgraph on the vertices $\{3,4,5,6,7,8\}$. Again we know that this subgraph must have a monochromatic triangle. If it were a blue monochromatic triangle, then we would get a blue copy of $K_4$ with the triangle and vertex $y$. So we must have a red triangle $\{a,b,c,\}$ with $a < b < c$ on these vertices. Now in order to avoud a red copy of the 3-pan with ordering $B$, we know that vertex $b$ has blue edges to vertices 1 and 2. But now we can take the subgraph on vertices $\{1,2,b,y\}$ and see that we get a blue copy of the 3-pan with ordering $B$.  Thus $r_{<_B}(\text{3-pan}) \leq 14$.
\end{proof}

\begin{figure}[h]
\centering
\includegraphics[width=6cm]{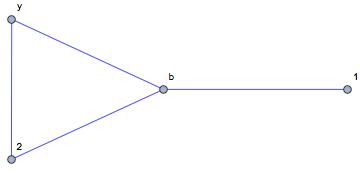}
\caption{Blue copy of the 3-pan with ordering $B$}
\end{figure}

Now we will consider the ordering of the 3-pan in which vertex 2 is attached to vertex 1, which we will refer to as ordering $C$ of the 3-pan.

\begin{figure}[h]
\centering
\includegraphics[width=6cm]{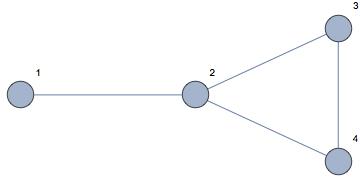}
\caption{Ordering $C$ of the 3-pan}
\end{figure}

\begin{Th}
The ordered Ramsey number of the 3-pan with ordering $C$ is bounded above by 11.
\end{Th}

\begin{proof}
Consider $K_{11}$. Initialize a list $Q=\{2,3,4,5,6,7\}$. Then we know that the subgraph induced by the vertices in $Q$ must contain a monochromatic triangle. Remove the lowest vertex of this triangle from $Q$ and add vertex 8 to $Q$. Then the subgraph of vertices now in $Q$ must also contain a monochromatic triangle. Again remove its lowest vertex and now add in 9. Continue this process twice more, adding in vertices 10 and 11. Then we will get five monochromatic triangles all with a different lowest vertex. And we know that in order to avoid a copy of the 3-pan with ordering $C$ none of these triangles can have an edge of the same color as the triangle to any of the vertices in $K_{11}$ lower than the triangle's lowest vertex.

Since there are five monochromatic triangles, we know that three of them must be the same color. Assume, without loss of generality, that we have three red triangles. Then let their lowest vertices be $x, y, z$ with $x < y < z$. Then we know that $z$ has a blue edge to $y$ and $x$, we know that $y$ has a blue edge to $x$ and we know that $x$ has a blue edge to vertex 1. Thus we get a blue copy of the 3-pan with ordering $C$.  Thus $r_{<_C}(\text{3-pan}) \leq 11$.
\end{proof}

\begin{figure}[h]
\centering
\includegraphics[width=6cm]{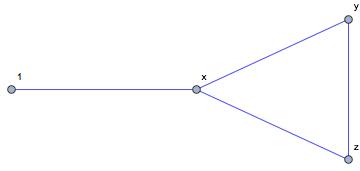}
\caption{Blue copy of the 3-pan with ordering $C$}
\end{figure}

\section{Ordered Ramsey Number of $K_n$ with a Pendant Edge}\label{sec:pendant}

Finally in this section we will be able to extend our proof of the upper bound on ordering $C$ of the 3-pan to the infinite family of graphs with a copy of $K_{n-1}$ on the vertices $\{2,3,...,n\}$ and with an edge between vertices 1 and 2. We make the following definition

\begin{Def}
The $\textit{complete with 1-pendant}$ ordering of a graph on $n$ vertices consists of a complete subgraph on vertices $\{2,3,...,n\}$ and an edge between vertex 1 and vertex 2. 
\end{Def}

For example, the complete with 1-pendant graph on 6 vertices is shown below in Figure 17

\begin{figure}[h]
\centering
\includegraphics[width=6cm]{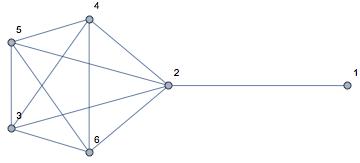}
\caption{Complete with 1-pendant graph on 6 vertices}
\end{figure}

\begin{Prop} \label{main} Let $K_n^1$ be the complete with 1-pendant graph on $n+1$ vertices. Then we have that $r_< (K_m^1,K_n^1) \leq R(m,n) + m + n -1$ where $R(m,n)$ is the standard Ramsey number for the complete graph on $m$ vertices vs. the complete graph on $n$ vertices. 
\end{Prop}

\begin{proof}
Start with $1+R(m,n)$ vertices. Then we have a set $V=\{2,3,...,1+R(m,n)\}$ which contains either a $K_m$ or $K_n.$ Remove the lowest ordered vertex, and replace it with $2+R(m,n).$ If we continue this process until we have added $(m+n-2)$ vertices, then we have $(m+n-1)$ copies of either $K_m$ or $K_n.$ Specifically, we have $n$ red copies of $K_m$ or $m$ blue copies of $K_n.$ Suppose we have $n$ red copies of $K_m.$ Then all of the copies have a unique lowest vertex. Now, we know that each copy's lowest vertex must have a blue edge to all other lower ordered vertices or else we would get a red copy of $K_m^1$ on $m+1$ vertices. But then if we arrange these $n$ lowest vertices along with vertex 1 in decreasing order, $v_n,v_{n-1},...,v_2,v_1,1,$ we know that each vertex has to have blue edges to all vertices after it, so we get a blue $K_{n+1}$, which inherently contains a $K_n^1.$ Thus, by Lemma 1.1, we get a blue copy of the complete with 1-pendant graph on $n+1$ vertices. Now instead suppose we have $n$ blue copies of $K_n.$ This proof follows the same idea for when we have $n$ copies of $K_m.$ Therefore, $r_{<}(K_m^1,K_n^1)\leq R(m,n) + m+n-1.$
\end{proof}

This bound is actually tight, which can be demonstrated by a construction that we found with David Conlon. 

\begin{Prop}\label{prop:lb}
There exists an edge coloring of the complete graph on $R(m,n)+m+n-2$ vertices that does not contain a monochromatic copy of $K_m^1$ in the first color nor of $K_n^1$ in the second color. 
\end{Prop}

\begin{proof}
Arbitrarily order the $R(m,n)+m+n-2$ vertices in the complete graph $G=K_{R(m,n)+m+n-2}.$ Take the highest ordered $R(m,n)-1$ vertices of $G.$ Then we know there is some way to color the edges among these vertices so that we do not get a red copy of a $K_m$ or a blue copy of a $K_n$. Color the edges this way. Call this subgraph of the highest ordered $R(m,n)-1$ vertices $Z.$ Now take the $n-1$ vertices before $Z$ in the ordering of $G$. Color all the edges among these vertices blue so that we get a blue $K_{n-1}.$ Call this subgraph $Y$. Now take the next $m-1$ vertices before $Y$ in the ordering of $G$. Color all the edges among these vertices red so that we get a red $K_{m-1}$. Call this subgraph $X$. Then the only vertex of $G$ that is not in $Z,Y,$ or $X$ is vertex $1$. Now color all the edges between $Z$ and $X$ blue. Color all the edges between $Y$ and $X$ blue. Color all the edges between $X$ to vertex $1$ red. Color the edges between $Z$ and $Y$ red. Color the edges between $Z$ to vertex $1$ red. Finally, color the edges between $Y$ to vertex $1$ blue. See Figure 1 below. Recall that $1< X < Y < Z$ where by $<$ we mean that all the vertices in one set have order less than all the vertices in the other set. We can see that this construction guarantees that whenever we get a red copy of $K_m$, there are no red edges from this $K_m$ to a lower vertex in $G.$ Similarly, whenever we get a blue copy of $K_n$, there are no blue edges from this $K_n$ to a lower vertex in $G.$ Thus, there is no red copy of a $K_m^1$ and no blue copy of a $K_n ^1$ on $m+n+1$ vertices.
\end{proof}


\begin{figure}[h]
\centering
\begin{tikzpicture} [scale = 1.3]

\draw[red, thick] (7,.15) -- (9.75,2.05);
\draw[red, thick] (7,.25) -- (9.7,2.15);
\draw[red, thick] (6.94,.35)-- (9.6,2.21);

\draw[red, thick] (10.3,2.1) -- (12.98,.15);
\draw[red, thick] (10.35,2.2) -- (12.99,.25);
\draw[red, thick] (10.3,2.4) -- (13.05,.35);

\draw[blue, thick] (12.97,-.23) -- (10.23,-2.1);
\draw[blue, thick] (13.1,-.24) -- (10.33,-2.15);
\draw[blue, thick] (13.2,-.3) -- (10.4,-2.25);

\draw[red, thick] (9.75,-2.2) -- (6.95,-.1);
\draw[red, thick] (9.7,-2.3) -- (6.9,-.2);
\draw[red, thick] (9.65,-2.4) -- (6.85,-.3);

\draw[blue, thick] (9.9,-1.95) -- (9.9,1.95);
\draw[blue, thick] (10,-1.95) -- (10,1.95);
\draw[blue, thick] (10.1,-1.95) -- (10.1,1.95);

\draw[blue, thick] (7,-.1) -- (13,-.1);
\draw[blue, thick] (7,0) -- (13,0);
\draw[blue, thick] (7,.1) -- (13,.1);

\draw[black] (6.7,0) circle (9.5pt) node[anchor=center] {$1$};
\draw[black] (10,2.3) circle (9.5pt) node[anchor=center] {Z} node[anchor=west] {\hspace{.5cm} $R(m,n)-1$};
\draw[blue] (13.3,0) circle (9.5pt) node[anchor=center] {Y}  node[anchor=west] [black] {\hspace{.5cm} $n-1$};
\draw[red] (10,-2.3) circle (9.5pt) node[anchor=center] {X} node[anchor=west] [black] {\hspace{.5cm} $m-1$};

\end{tikzpicture}
\caption{Edge colorings between the sets $X$, $Y$, $Z$, and 1}
\end{figure}
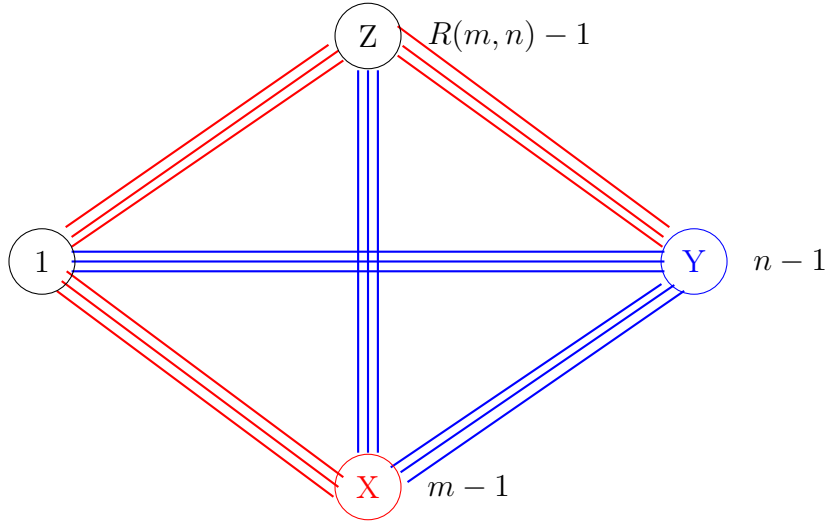


\begin{Th}
 We have $r_{<}(K_m^1,K_n^1) =  R(m,n) + m+n-1$.
\end{Th}

\begin{proof}
The proof follows immediately from combining Propositions \ref{main} and .
\end{proof}

\section{SAT Solver Results}\label{sec:SAT}

We use the SAT solver Minisat to determine the ordered Ramsey numbers for all orderings of $K_4-e$, including the off-diagonal cases.  We introduce for each edge $ij$ a boolean variable $x_{i,j}$. We interpret assigning $x_{i,j}$ to TRUE as coloring $ij$ red, and interpret assigning $x_{i,j}$ to FALSE as coloring $ij$ blue.  Then we naturally express the condition ``$K_N$ has no monochromatic ordered $K_4-e$'' as a boolean formula in conjunctive normal form (CNF), which is to say, a conjunction of disjunctions (an ``\emph{and} of \emph{or}s''). 

For example, the following CNF forbids,  in each color, a monochromatic $K_4-e$ with the missing edge being between the lowest two vertices:
\[
    \Phi_{0,1} = \bigwedge_{0\leq i<j<k<\ell<N} (x_{i,k} \vee  x_{i,\ell} \vee  x_{j,k} \vee  x_{j,\ell} \vee  x_{k,\ell} ) \wedge (\neg x_{i,k} \vee  \neg x_{i,\ell} \vee  \neg x_{j,k} \vee  \neg x_{j,\ell} \vee  \neg x_{k,\ell} )
\]

The formula $\Phi$ above is satisfiable for $N\leq 11$ and unsatisfiable for $N\geq 12$.  Satisfying assignments correspond to colorings without any monochromatic ordered $K_4 - e$.  Hence the Ramsey number is 12.

\begin{table}[hbt]
    \centering
    \begin{tabular}{c|c|c|c|c|c|c}
         & 0-1  & 0-2  & 0-3  & 1-2  & 1-3  & 2-3 \\
         \hline
    0-1   & 12  & 14  & 14  & 13  & 14  & 13 \\
    0-2     &   & 14  & 15  & 14  & 15  & 14 \\
    0-3     &   &   & 15  & 14  & 15  & 14 \\
    1-2    &   &   &   &  13 & 14  & 13 \\
    1-3     &   &   &   &   &  14 & 14 \\
    2-3     &   &   &   &   &   & 12 \\
    \end{tabular}
    \caption{Table of $r_<(K_4 - e, K_4-e)$, with missing edge indicated}
    \label{tab:my_label}
\end{table}

\section{Summary}
In this paper we were able to completely determine the ordered Ramsey numbers for every ordering of $K_2 \cup K_2$, for every ordering of $K_4 - e$,  and for the complete with 1-pendant graph on any number of vertices (relative to the classical Ramsey number). The latter result is particularly interesting considering it is often difficult to prove exact results using Ramsey numbers that are not exactly known themselves, i.e. the fact that the ordered Ramsey number depends on $R(n)$. 

An idea for future work  would be to try to determine \textit{why} certain orderings of a graph give different ordered Ramsey numbers than others; the work in \cite{balko} on paths suggests that the more ``monotonic'' the ordering, the larger the Ramsey number will be. 

\section{Acknowledgements}

The first author would like to thank Professor Conlon for reviewing preliminary results and introducing him to ordered Ramsey numbers. We would also like to thank him for his work on the construction in Theorem 5.3. 

The last three authors recognize the majority contribution of the first author by breaking the tradition of alphabetical ordering of the authors.

\end{document}